\documentclass[reqno]{amsart}
\usepackage{amssymb,amsmath,amsthm,amscd,latexsym,amsfonts}
\usepackage{mathtools}
\usepackage[english]{babel}
\usepackage[T1]{fontenc}
\usepackage{inputenc}
\usepackage[arrow,matrix]{xy}
\usepackage{amsaddr}
\usepackage{graphicx}
\usepackage{cite}
\newtheorem{thm}{Theorem}[section]

\newtheorem{prop}[thm]{Proposition}

\theoremstyle{definition}
\newtheorem{defn}[thm]{Definition}
\newtheorem{exam}[thm]{Example}
\theoremstyle{remark}

\numberwithin{equation}{section}

\begin{document}

\title[Leibniz algebras whose  solvable ideal is the maximal extension of the nilradical]{Leibniz algebras whose  solvable ideal is the maximal extension of the nilradical}
\author{K.K.~Abdurasulov\textsuperscript{1}, Z.Kh.~Shermatova\textsuperscript{1}}

\address{\textsuperscript{1}V.I. Romanovskiy Institute of Mathematics, Uzbekistan Academy of Sciences, University street, 9,  Tashkent, 100174, Uzbekistan}

\email{abdurasulov0505@mail.ru}
\email{z.shermatova@mathinst.uz}

\begin{abstract} The paper is devoted to the so-called complete Leibniz algebras.
  It is known that a Lie algebra with a complete ideal is split.
 We  will prove that this result is valid for   Leibniz algebras  whose complete ideal is a solvable  algebra such that the codimension of
 nilradical is equal to the number of generators
of the nilradical.
 \end{abstract}

\subjclass[2010] {17A32, 17A36, 17B30, 17B56.}

\keywords{Leibniz algebra, complete algebra, semisimple algebra,
 radical, nilradical, derivation, inner derivation.}

\maketitle

\section{Introduction}

Lie algebras, which all derivations are inner and the center is zero, appeared in M. Goto's work \cite{Goto}, although at that time, there were few algebras of this kind, therefore, their meaning could not be understood very well. The first important result on complete Lie algebras was given in \cite{Schenkman}, in the context of Schenkman's theory of subinvariant Lie algebras. The term of "complete Lie algebras" was firstly used by N. Jacobson \cite{jacobson}, it was proved that arbitrary Lie algebra with a complete ideal is split. In recent years, various authors have focused on the classification and structural properties of complete Lie algebras. In \cite{Meng} D.J. Meng has proved that any finite-dimensional complete Lie algebra can be decomposed into the direct sum of simple complete ideals, and such a decomposition is unique up to the order of the ideals. Also, Y.C. Gao and D.J. Meng \cite{Gao} first gave a necessary and sufficient condition for some solvable Lie algebras with $l$-step nilradicals to be complete and a method for constructing non-solvable complete Lie algebras.

Some scientists used different approaches to introduce the notion of complete Leibniz algebras, and later it was accepted that a Leibniz algebra is called complete if its center is trivial and all derivations are inner, as in Lie algebras, by J.M. Ancochea and R.Campoamor-Stursberg \cite{Ancochea}. Authors proved the completeness of the solvable Leibniz algebra with the null-filiform nilradical. Such algebras are also important from a cohomological point of view and many complete Leibniz algebras belong to the class of cohomologically rigid algebras.  Later J.K. Adashev and others \cite{Adashev1} showed that solvable Leibniz algebra with nilradical being naturally graded $p$-filiform Leibniz algebra of maximal codimension is complete. Moreover, B.A. Omirov and U.X. Mamadaliyev \cite{Mamadaliyev} classified complete solvable Leibniz algebras with the nilradical of an arbitrary characteristic sequence. Also,   K.K. Abdurasulov \cite{Abdurasulov} proved the completeness of a solvable Leibniz algebra with the nilradical of codimension equals the number of generators of the nilradical.

The aim of the present article is  extend some results obtained for complete Lie algebras
to the Leibniz algebras case. We construct some Leibniz algebras with complete
radical in this work.
It is known that a Lie algebra with a complete ideal is split. We will prove that this result is valid for   Leibniz algebras  whose complete ideal is a solvable Leibniz algebra such that the codimension of
 nilradical is equal to the number of generators
of the nilradical.

Throughout the paper we denote  by $L$ a finite-dimensional Leibniz algebra over the field of complex numbers.

\section{Preliminaries}

In this section we present some necessary definitions and preliminary results  which are used in  this paper.

\begin{defn}\cite{Loday} A  \textit{Leibniz algebra} $L$ is a vector space  equipped with a bilinear map $[-,-] : L\times L \rightarrow L$ (multiplication) satisfying the Leibniz identity
\[ \big[x,[y,z]\big]=\big[[x,y],z\big] - \big[[x,z],y\big] \]
for all $x,y,z\in L.$
\end{defn}

The theory of Leibniz algebras has developed very intensively in many different directions.
Some of the results of this theory were presented in the book \cite{AyupovOmirovRakhimov}.


For a given Leibniz algebra $L$ we define the following two-sided ideal
$$Z(L) =\{x \in L \mid [x,y]=[y,x] = 0,\ \text{for \ all}\ y \in L \}$$
called 
the \emph{center} of $L.$

\label{der}  A linear map $d: L\rightarrow L$ on a Leibniz algebra $L$ is said to be a {\it derivation} if
for all $x, y\in L,$ the following condition holds:
$$d([x,y])=[d(x),y] + [x, d(y)].$$

The set of all derivations of $L$ (denoted by $Der(L)$) forms a Lie algebra with respect to the commutator $[d, d'] = d \circ d' - d' \circ d$ for $d,d' \in Der(L)$.

Note that the operator of right multiplication by an element $x\in L$ (further denoted by $R_x$) is a derivation, which is called {\it inner derivation}. Derivations which are not inner are called {\it outer derivations.}
The set $Inner(L) = \{R_x : x \in L\}$  of all inner derivations forms a Lie algebra
with respect to the  commutator, moreover, the following identity holds:
$$R_xR_y-R_yR_x=R_{[y,x]}.$$


\begin{defn} A Lie algebra $\mathfrak{g}$ is called  \emph{complete}
  if  $Z(\mathfrak{g})=\{0\}$ and all derivations of $\mathfrak{g}$ are
inner.
\end{defn}

\begin{exam} a) Semisimple Lie algebras over a field of characteristic $0$ are
complete.

b) The (unique) non-abelian Lie algebra of dimension $2$ is  complete.

\end{exam}

\begin{defn} \cite{Meng} A complete Lie algebra $\mathfrak{g}$ is called a simple complete Lie algebra, if any non-trivial ideal of $\mathfrak{g}$ is not complete.
\end{defn}

Obviously, simple Lie algebra over the field of characteristic 0 and the non-abelian
Lie algebra of dimension 2 are simple complete Lie algebras.

\begin{prop}\cite{jacobson} \label{prop11}
If $R$ is complete  and an ideal in a Lie algebra $\mathfrak{g}$, then $\mathfrak{g}=R\oplus S,$ where $S$ is an ideal.
\end{prop}

By analogy with Lie algebras, now we give the definition of complete Leibniz algebra that was first introduced in \cite{Ancochea}.
\begin{defn} \cite{Ancochea} \label{def}  A Leibniz algebra $L$ is called \textit{complete} if $Z(L)=\{0\}$ and all derivations of $L$ are
inner.
\end{defn}

Now we present the notion of simple complete Leibniz algebra.
\begin{defn} A complete Leibniz algebra $L$ is called a simple complete Leibniz algebra, if any non-trivial ideal of $L$ is not complete.
\end{defn}

We define the following sequences:
\[L^1=L, \ L^{k+1}=[L^k,L],  \ k \geq 1, \qquad \qquad
L^{[1]}=L, \ L^{[k+1]}=[L^{[k]},L^{[k]}], \ k \geq 1,\]
so-called the  \emph{lower central} and the \emph{derived series} of $L$, respectively.

\begin{defn} A Leibniz algebra $L$ is called
\emph{nilpotent} (respectively, \emph{solvable}), if there exists $s\in\mathbb N$ ($m\in\mathbb N$) such that $L^{s}=\{0\}$ (respectively, $L^{[m]}=\{0\}$).
 The minimal number $s$ (respectively, $m$) with such property is called index of nilpotency (respectively, index of solvability) of $L.$
\end{defn}

The maximal nilpotent (respectively, solvable) ideal of a Leibniz algebra is said to be the \emph{nilradical} (respectively, \emph{radical})
of the  algebra.

In \cite{Barnes} D. Barnes proved an analogue of Levi's theorem for Leibniz algebras; namely, an arbitrary Leibniz algebra $L$ is decomposed into a semidirect sum of its solvable radical and a semisimple Lie subalgebra, $L=S\dot{+} R$. The subalgebra $S,$ similarly to Lie algebras theory, is called a \emph{Levi subalgebra} of the Leibniz algebra $L$.

\begin{prop}\cite{Gorbachev} \label{radnil} Let $R$ be the radical of a Leibniz algebra $L,$ and $N$ be its nilradical. Then
$[L,R]\subset N.$
\end{prop}

Let $R$ be a solvable Leibniz algebra with the nilradical $N.$ We denote by $Q$ the complementary vector space of the nilradical $N$ to the algebra $R.$
For a solvable
Leibniz algebra with a given nilradical the restriction of the right multiplication operator on elements
of the subspace complementary to the nilradical is non-nilpotent outer derivation of the nilradical.
Therefore, the dimension of the subspace complementary to the nilradical is not greater than maximal
number of nil-independent outer derivations of the nilradical of the solvable Leibniz algebra.

For a finite-dimensional nilpotent Leibniz algebra  $N$ and for the matrix of the linear operator $R_x,$ denote by $C(x)$ the
descending sequence of its Jordan blocks' dimensions with an arbitrary element from the set $N \setminus N^2.$ Consider the
lexicographical order on the set of such sequence, i.e. $C(x)=(n_1,n_2,\ldots,n_k)\leq C(y)=(m_1,m_2,\ldots,m_s)$
 if and only if there exists $i\in \mathbb{{N}}$ such that $n_j=m_j$ for any $j<i$ and $n_i<m_i.$


\begin{defn} The sequence
$$C(N)=\max\limits_{x\in N\setminus N^2} C(x)  $$ is said to be the
characteristic sequence of the nilpotent Leibniz algebra $N.$
\end{defn}

Let $N$ be a nilpotent Leibniz algebra with the characteristic sequence
$(m_1,\ldots,m_s)$ and multiplication table
$$N_{m_1,\dots,m_s}: [e^t_i,e^1_1]=e^t_{i + 1},\ 1 \leq t\leq s,\, 1\leq i\leq m_t-1.$$

 Consider a solvable Leibniz algebra $R$ with the nilradical $N_{m_1,\dots,m_s}$ and codimension of nilradical is $s.$ The following theorem asserts that such a Leibniz algebra is unique up to isomorphism.

\begin{thm}\cite{Mamadaliyev} An arbitrary solvable Leibniz algebra with the nilradical $N_{m_1,\dots,m_s}$ of codimension $s$ is isomorphic to the algebra:
 $$R(N_{m_1,\dots,m_s},s): \left\{\begin{array}{ll}
[e^t_i,e^1_1]=e^t_{i+1}, & 1\leq t\leq s, \,1\leq i\leq m_t-1 ,\\[1mm]
[e^1_i,x_1]=ie^1_i, & 1\leq i\leq m_1,\\[1mm]
[e^t_i,x_1]=(i-1)e^t_i, & 2\leq t\leq s, \,2\leq i\leq m_t,\\[1mm]
[e^t_i,x_t]=e^t_i, & 2\leq t\leq s, \,1\leq i\leq m_t,\\[1mm]
[x_1,e^1_1]=-e^1_1, \\[1mm]
 \end{array}\right.$$
where $\{x_1,\ldots x_s\}$ is a basis of the complementary vector space.
\end{thm}

It should be noted that the Leibniz algebra $R(N_{m_1,\dots,m_s},s)$ is also complete \cite{Mamadaliyev}.

Now as an analogue of Proposition \ref{prop11}, we give the following conjecture for complete Leibniz algebra.

\textbf{Conjecture.} \cite{Complete} If a complete Leibniz algebra $R$ is an ideal in the Leibniz algebra $L$, then $L=R\oplus S,$ where $S$ is an ideal in $L$.


In \cite{Complete} it is shown that  the above conjecture is true for a Leibniz algebra whose semisimple part is $\mathfrak{sl_2}$ and the complete ideal is a solvable Leibniz algebra with the nilradical $N_{m_1,\dots,m_s}.$

\begin{thm} \cite{Complete} Let $L$ be a Leibniz algebra such that $L=R\dot{+}\mathfrak{sl_2}$ be its Levi decomposition,  where $R$ is a complete solvable ideal $R(N_{m_1,\dots, m_s},s),$
then $L=R\oplus \mathfrak{sl_2}.$ In other words, $L$ is the direct sum of ideals.
\end{thm}

In this work we will prove that the previous conjecture is valid for Leibniz algebras whose its complete ideal is a solvable Leibniz algebra with the nilradical of the maximal codimension.

\section{Leibniz algebras whose  complete ideal is a solvable algebra with the nilradical of the
maximal codimension}\label{Par3.3}

%

In this section  we show that if $L$ is a Leibniz algebras which has $\mathfrak{sl_2} $ as a semisimple part,   and  its complete ideal is a  solvable Leibniz algebra such that the codimension
of the nilradical is equal to the number of generators of the nilradical, then $L$ is the direct sum of ideals.

We denote by ${\mathcal A(k)}$ the $k$-dimensional abelian algebra,  by $R({\mathcal A(k)},s)$ the class of solvable Leibniz algebras with the nilradical ${\mathcal A(k)}$
of codimension  $s.$

Next theorem is devoted to the classification of solvable Leibniz algebras with the nilradical ${\mathcal A(k)}$ of the  maximal dimension.
 \begin{thm}\cite{Adashev} The maximal possible dimension of algebras of the family $R({\mathcal A(k)}, s)$ is equal to $2k$, that is, $s=k$.
Moreover, an arbitrary algebra of the family $R({\mathcal A(k)}, k)$ is decomposed into a direct sum of copies of two-dimensional non-trivial solvable Leibniz algebras.
\end{thm}
Then there exists a basis $\{f_1, f_2, \dots, f_k, x_1, x_2, \dots, x_k\}$ of the family of algebras $R({\mathcal A(k)}, k)$ such that the multiplication table has the form:
$$[f_i,x_i]=f_{i},\quad
[x_i,f_i]=\alpha_if_i, \quad 1\leq i\leq k,$$
where $\alpha_{i}\in\{-1,0\}.$

From \cite{Adashev} it is known that all derivations of an arbitrary algebra of the family $R({\mathcal A(k)}, k)$  are  inner
 and the center of the algebra is trivial, so  an arbitrary solvable algebra  from the family  $R({\mathcal A(k)}, k)$ is complete.

We give the following proposition  and use it to prove  the next theorem.
\begin{prop}\label{Abelianprop} Let $L$ be a Leibniz algebra such that $L=R\dot{+}\mathfrak{sl_2} $ is its
Levi decomposition, where the complete solvable ideal $R$ is  from the family of algebras $R({\mathcal A(k)}, k).$
Then $L=R\oplus \mathfrak{sl_2} $ in other words, $L$ is the direct sum of ideals.
\end{prop}
\begin{proof}
We  will prove that  the relation
\begin{equation}\label{eq1}
 [R,\mathfrak{sl_2} ]=[\mathfrak{sl_2} ,R]=0
 \end{equation}  is valid.
There exists a basis $\{e,f,h, f_1, f_2, \dots, f_k, x_1, x_2, \dots, x_k\}$  of the algebra $L,$ due to the Proposition \ref{radnil}, such that the products
$[R,\mathfrak{sl_2} ]$ and $[\mathfrak{sl_2} ,R]$ have the
following form:

$$\begin{array}{lll}
\, [e,f_i]=\sum\limits_{j=1}^{k}\alpha_{i,j}f_j,&
[h,f_i]=\sum\limits_{j=1}^{k}\beta_{i,j}f_j, & [f,f_i]=\sum\limits_{j=1}^{k}\gamma_{i,j}f_j,\\
\,[f_i,e]=\sum\limits_{j=1}^{k}\alpha{'}_{i,j}f_j,&
 [f_i,h]=\sum\limits_{j=1}^{k}\beta{'}_{i,j}f_j, &[f_i,f]=\sum\limits_{j=1}^{k}\gamma{'}_{i,j}f_j,\\

\, [e,x_i]=\sum\limits_{j=1}^{k}\delta_{i,j}f_j,& [h,x_i]=\sum\limits_{j=1}^{k}\sigma_{i,j}f_j, & [f,x_i]=\sum\limits_{j=1}^{k}\tau_{i,j}f_j,\\
\,[x_i,e]=\sum\limits_{j=1}^{k}\delta{'}_{i,j}f_j,&
[x_i,h]=\sum\limits_{j=1}^{k}\sigma{'}_{i,j}f_j, &[x_i,f]=\sum\limits_{j=1}^{k}\tau{'}_{i,j}f_j,\\
\, & & 1\leq i\leq k. \\
\end{array}$$

We consider Leibniz identity for $1\leq i\leq k,$
$$[f_i,[e,x_i]]=[[f_i,e],x_i]-[[f_i,x_i],e]=[\sum\limits_{j=1}^{k}\alpha{'}_{i,j}f_j,x_i]-[f_i,e]=$$
$$=\alpha{'}_{i,i}f_i-\sum\limits_{j=1}^{k}\alpha{'}_{i,j}f_j=
-\sum\limits_{j=1,\ j\neq i}^{k}\alpha{'}_{i,j}f_j=[f_i,\sum\limits_{j=1}^{k}\delta_{i,j}f_j]=0,$$
thus  $\alpha{'}_{i,j}=0$ for $  1\leq j\leq k, \ j\neq i. $
$$[e,[x_i,x_j]]=0=[[e,x_i],x_j]-[[e,x_j],x_i]=[\sum\limits_{t=1}^{k}\delta_{i,t}f_t,x_j]-[\sum\limits_{t=1}^{k}\delta_{j,t}f_t,x_i]=$$
$$=\delta_{i,j}f_j-\delta_{j,i}f_i,$$
thus $\delta_{i,j}=0$ for $1\leq i,j\leq k,\ j\neq i.$

From $[f_i,x_i]+[x_i,f_i]\in Ann_r(L)$ we obtain the following relations for $1\leq i\leq k$
$$[e,[f_i,x_i]+[x_i,f_i]]=0=[e,(1+\alpha_i)f_i]=(1+\alpha_i)[e,f_i],$$
if $\alpha_i=0$ for some $1\leq i\leq k,$ then  $\alpha_{i,j}=0$ for any $1\leq j\leq k.$

Next we consider  the Leibniz identity for  $\alpha_i=-1,$ where $1\leq i\leq k:$
$$[e,[f_i,x_i]]=[e,f_i]=\sum\limits_{j=1}^{k}\alpha_{i,j}f_j=[[e,f_i],x_i]-[[e,x_i],f_i]=[\sum\limits_{j=1}^{k}\alpha_{i,j}f_j,x_i]-$$
$$-[\sum\limits_{j=1}^{k}\delta_{i,j}f_j,f_i]=
\alpha_{i,i}f_i,$$
thus  $\alpha_{i,j}=0$ for $1\leq j\leq k, \ j\neq i. $
$$[[x_i,e],f_i]=[\sum\limits_{j=1}^{k}\delta{'}_{i,j}f_j,f_i]=0=
[x_i,[e,f_i]]+[[x_i,f_i],e]=[x_i,\alpha_{i,i}f_i]-[f_i,e]=$$$$=-\alpha_{i,i}f_i-\alpha{'}_{i,i}f_i=
-(\alpha_{i,i}+\alpha{'}_{i,i})f_i,$$
we have $\alpha{'}_{i,i}=-\alpha_{i,i}.$
$$[x_i,[e,x_j]] = [[x_i,e],x_j]-[[x_i,x_j],e]=[\sum\limits_{t=1}^{k}\delta{'}_{i,t}f_t,x_j]=\delta{'}_{i,j}f_j=$$$$=[x_i,\delta_{j,j}f_j]=\delta_{j,j}[x_i,f_j],$$
thus  $$\delta{'}_{i,j}=0, \ 1\leq i,j\leq k, \ j\neq i,$$

if $\alpha_i=-1$ for some $1\leq i\leq k,$ then $\delta{'}_{i,i}=-\delta_{i,i},$

if $\alpha_i=0$ for some $1\leq i\leq k,$ then $\delta{'}_{i,i}=0.$

Similarly, if we  replace $e$  alternately by $h$ and $f$ in the above Leibniz identity, we obtain
for some  $1\leq i\leq k,$ such  that $\alpha_i=-1$
$$[e,f_i]=\alpha_{i,i}f_i, \ [f_i,e]=-\alpha_{i,i}f_i, \ [e,x_i]=\delta_{i,i}f_i, \ [x_i,e]=-\delta_{i,i}f_i,$$
$$[h,f_i]=\beta_{i,i}f_i, \ [f_i,h]=-\beta_{i,i}f_i, \ [h,x_i]=\sigma_{i,i}f_i, \ [x_i,h]=-\sigma_{i,i}f_i,$$
$$[f,f_i]=\gamma_{i,i}e_i, \ [f_i,f]=-\gamma_{i,i}f_i, \ [f,x_i]=\tau_{i,i}f_i, \ [x_i,f]=-\tau_{i,i}f_i,$$
for some  $1\leq i\leq k,$ such  that $\alpha_i=0$
$$[e,f_i]=0, \ [f_i,e]=\alpha{'}_{i,i}f_i, \ [e,x_i]=\delta_{i,i}f_i, \ [x_i,e]=0,$$
$$[h,f_i]=0, \ [f_i,h]=\beta{'}_{i,i}f_i, \ [h,x_i]=\sigma_{i,i}f_i, \ [x_i,h]=0,$$
$$[f,f_i]=0, \ [f_i,f]=\gamma{'}_{i,i}f_i, \ [f,x_i]=\tau_{i,i}f_i, \ [x_i,f]=0.$$

Using the latter equalities in combination with the following
$$[f_i,[e,h]] = [[f_i,e],h]-[[f_i,h],e] =[-\alpha_{i,i}f_i,h]-[-\beta_{i,i}f_{i},e]=$$
$$=\alpha_{i,i}\beta_{i,i}f_i-\beta_{i,i}\alpha_{i,i}f_i=0=
[f_i,2e]=-2\alpha_{i,i}f_{i},$$
$$[x_i,[e,h]]=0=[x_i,2e]=-2\delta_{i,i}f_{i},$$
we get $\alpha_{i,i}=\delta_{i,i}=0,$ analogically, $\beta_{i,i}=\sigma_{i,i}=\gamma_{i,i}=\tau_{i,i}=0$  for some  $1\leq i\leq k,$  such  that  $\alpha_i=-1.$

In a similar way from the Leibniz identity for
$[f_i,[e,h]$ and $[[e,h],f_i],$ we derive that $\alpha{'}_{i,i}=0,$  $\delta_{i,i}=0$  for some  $1\leq i\leq k,$  such  that  $\alpha_i=0,$ also,
$\beta{'}_{i,i}=\sigma_{i,i}=\gamma{'}_{i,i}=\tau_{i,i}=0.$

It follows that  the relation (\ref{eq1}) is valid, i.e. $L=R\oplus \mathfrak{sl_2},$
 which completes the proof of proposition.
\end{proof}

Furthermore, by $\tau$ we denote the nilindex of $N$, that is, $N^{\tau-1}\neq0$ and $N^{\tau}=0$.

The description of  solvable Leibniz algebras for which the codimension
of nilradical is equal to the number of generators of the nilradical is given in \cite{Abdurasulov}.
\begin{thm} \label{thm35}\cite{Abdurasulov} Let $R=N\oplus Q$ be a solvable Leibniz algebra such that $dimQ=dim (N/N^2)=k.$ Then $R$ admits a
basis $\{e_1, e_2, \dots, e_n, x_1, x_2, \dots, x_k\}$ such that the table of multiplication in $R$ has the following
form:
\begin{equation}\label{table1}
\left\{\begin{array}{ll}
[e_i,e_j]=\sum\limits_{t=k+1}^{n}c_{i,j}^te_t,& 1\leq i, j\leq n,\\[1mm]
[e_i,x_i]=e_i,& 1\leq i\leq k,\\[1mm]
[x_i,e_i]=(b_i-1)e_i, & b_i \in \{0,1\}, 1\leq i\leq k,\\[1mm]
[e_i,x_j]=a_{i,j}e_i,& k+1\leq i\leq n,\ \ 1\leq j\leq k,\\[1mm]
[x_j,e_i]=\sum\limits_{t=1}^{q}b_{j,i}^{i_t}e_{i_t}, &  k+1\leq i\leq n, \ 1\leq j\leq k,\\[1mm]
\end{array}\right.
\end{equation}
where omitted products are equal to zero and $a_{i,j}$ is the number of entries of a generator basis element $e_j$ involved in forming of non generator basis element $e_i$.
\end{thm}
It should be noted that the solvable Leibniz algebra $R$ is given in Theorem \ref{thm35} is also complete \cite{Abdurasulov}.

In the following  theorem we prove that if a Leibniz algebra with semisimple part  $\mathfrak{sl_2},$ and whose complete ideal is a solvable Leibniz algebra such that the codimension of
 its nilradical is equal to the number of generators
of the nilradical  then it satisfies the above conjecture.
\begin{thm} Let $L$ be a Leibniz algebra such that $L=R\dot{+}\mathfrak{sl_2} $ is its
Levi decomposition, which the complete solvable ideal $R$ is given in Theorem \ref{thm35}.
Then $L=R\oplus \mathfrak{sl_2} $ in other words, $L$ is the direct sum of ideals.
\end{thm}
\begin{proof}
Let $L$ be a Leibniz algebra and $L=R\dot{+}\mathfrak{sl_2} $ be its
Levi decomposition, where $R$ is a solvable ideal which the codimension of
the nilradical is equal to the number of generators of the
nilradical.
%
Then we  will show that the relation (\ref{eq1}) is satisfied.

First we note that  for any $m,$ the space $N^m$ forms an ideal of $R.$  In order to prove the theorem, we will prove by  induction on $m$  ($2\leq m \leq \tau$) that the quotient algebra $R/N^{m}$ admits a basis such that the following products are valid:
\begin{equation}\label{eq3}
[R/N^m,\mathfrak{sl_2} ]=[\mathfrak{sl_2} ,R/N^m]=0.
\end{equation}

For $m=2$ the solvable quotient Leibniz algebra $R/N^2$ has an abelian nilradical. Due to the Proposition \ref{Abelianprop} we conclude that the  products (\ref{eq3}) are valid.

Let us assume that the products (\ref{eq3}) are true for all $m$. In the quotient algebra $R/N^{m+1}$ we need to consider the products $[R/N^{m+1},\mathfrak{sl_2} ]$ and $[\mathfrak{sl_2},R/N^{m+1}].$
Then  the products $[R/N^{m+1},\mathfrak{sl_2} ]$ and $[\mathfrak{sl_2} ,R/N^{m+1}]$ have the
following form: for the sake of convenience further we will use congruences without indicating modulo $N^{m+1}$.

$$\begin{array}{lll}
\, [e,e_i]=\sum\limits_{j=1}^{n_{m-1}}\alpha_{i,j}e_{k+n_1+\ldots+n_{m-2}+j},&
[h,e_i]=\sum\limits_{j=1}^{n_{m-1}}\beta_{i,j}e_{k+n_1+\ldots+n_{m-2}+j}, & [f,e_i]=\sum\limits_{j=1}^{n_{m-1}}\gamma_{i,j}e_{k+n_1+\ldots+n_{m-2}+j},\\
\,[e_i,e]=\sum\limits_{j=1}^{n_{m-1}}\alpha{'}_{i,j}e_{k+n_1+\ldots+n_{m-2}+j},&
 [e_i,h]=\sum\limits_{j=1}^{n_{m-1}}\beta{'}_{i,j}e_{k+n_1+\ldots+n_{m-2}+j}, &[e_i,f]=\sum\limits_{j=1}^{n_{m-1}}\gamma{'}_{i,j}e_{k+n_1+\ldots+n_{m-2}+j},\\

\, [e,x_i]=\sum\limits_{j=1}^{n_{m-1}}\delta_{i,j}e_{k+n_1+\ldots+n_{m-2}+j},& [h,x_i]=\sum\limits_{j=1}^{n_{m-1}}\sigma_{i,j}e_{k+n_1+\ldots+n_{m-2}+j}, & [f,x_i]=\sum\limits_{j=1}^{n_{m-1}}\tau_{i,j}e_{k+n_1+\ldots+n_{m-2}+j},\\
\,[x_i,e]=\sum\limits_{j=1}^{n_{m-1}}\delta{'}_{i,j}e_{k+n_1+\ldots+n_{m-2}+j},&
[x_i,h]=\sum\limits_{j=1}^{n_{m-1}}\sigma{'}_{i,j}e_{k+n_1+\ldots+n_{m-2}+j}, &[x_i,f]=\sum\limits_{j=1}^{n_{m-1}}\tau{'}_{i,j}e_{k+n_1+\ldots+n_{m-2}+j},\\
\, & & 1\leq i\leq k+n_1+\ldots+n_{m-2}+n_{m-1}, \\
\end{array}$$
where $\{e_1,e_2,\ldots,e_k,e_{k+1},\ldots,e_{k+n_1},\ldots,e_{k+n_1+n_2+\ldots+n_{m-2}+1},\ldots,e_{k+n_1+n_2+\ldots+n_{m-2}+n_{m-1}},...,e_{k+n_1+n_2+\ldots+n_{\tau-2}}\}$ is a basis of $N$ and
$\{e_{k+n_1+\ldots+n_{m-2}+1},e_{k+n_1+\ldots+n_{m-2}+2},\ldots,e_{k+n_1+\ldots+n_{m-2}+n_{m-1}}\}\in N^{m}\setminus N^{m+1},$  $dim N=k+n_1+n_2+\ldots+n_{\tau-2}=n.$

 We consider  Leibniz identity for $1\leq i,l\leq k$, $l\neq i$
$$[e,[e_i,x_l]]=0=[[e,e_i],x_l]-[[e,x_l],e_i]=[\sum\limits_{j=1}^{n_{m-1}}\alpha_{i,j}e_{k+n_1+\ldots+n_{m-2}+j},x_l]-$$
$$-[\sum\limits_{j=1}^{n_{m-1}}\delta_{l,j}e_{k+n_1+\ldots+n_{m-2}+j},e_i]=\sum\limits_{j=1}^{n_{m-1}}\alpha_{i,j}a_{k+n_1+\ldots+n_{m-2}+j,l}e_{k+n_1+\ldots+n_{m-2}+j}.$$
Note that for any $i,$ $k+1\leq i \leq n$ there exists $j,$ $1\leq j \leq k$ such that $a_{i,j}\neq 0.$
Thus  $\alpha_{i,j}=0$ for $1\leq j\leq n_{m-1}.$
$$[e_i,[e,x_l]]=[e_i,\sum\limits_{j=1}^{n_{m-1}}\delta_{l,j}e_{k+n_1+\ldots+n_{m-2}+j}]=0=[[e_i,e],x_l]-[[e_i,x_l],e]=$$
$$=[\sum\limits_{j=1}^{n_{m-1}}\alpha{'}_{i,j}e_{k+n_1+\ldots+n_{m-2}+j},x_l]
=\sum\limits_{j=1}^{n_{m-1}}\alpha{'}_{i,j}a_{k+n_1+\ldots+n_{m-2}+j,l}e_{k+n_1+\ldots+n_{m-2}+j},$$
so  $\alpha{'}_{i,j}=0$ for $1\leq j\leq n_{m-1}.$ We have that $[e,e_i]=[e_i,e]=0,$ for $1\leq i \leq k,$ and deduce that
$[e,e_i]=[e_i,e]=0,$ for $k+1\leq i \leq k+n_1+\ldots+n_{m-1}.$ Similarly, $[h,e_i]=[e_i,h]=[f,e_i]=[e_i,f]=0,$ for $1\leq i \leq k+n_1+\ldots+n_{m-1}.$

Using the latter equalities in combination with the following
$$[x_i,[e,h]]=[x_i,2e]=2[x_i,e]=[[x_i,e],h]-[[x_i,h],e]=$$$$=[\sum\limits_{j=1}^{n_{m-1}}\delta{'}_{i,j}e_{k+n_1+\ldots+n_{m-2}+j},h]-
[\sum\limits_{j=1}^{n_{m-1}}\sigma{'}_{i,j}e_{k+n_1+\ldots+n_{m-2}+j},e]=0,$$
$$[[e,h],x_i]=[2e,x_i]=[e,[h,x_i]]+[[e,x_i],h]=$$$$=[e,\sum\limits_{j=1}^{n_{m-1}}\sigma_{i,j}e_{k+n_1+\ldots+n_{m-2}+j}]+
[\sum\limits_{j=1}^{n_{m-1}}\delta_{i,j}e_{k+n_1+\ldots+n_{m-2}+j},h]=0,$$ we get $[x_i,e]=[e,x_i]=0,$ for $1\leq i \leq k+n_1+\ldots+n_{m-1}.$

Finally, from the Leibniz identity for $[x_i,[h,f]],$ $[[h,f],x_i],$ $[x_i,[e,f]]$ and $[[e,f],x_i]$ we deduce that $[x_i,h]=[h,x_i]=[x_i,f]=[f,x_i]=0,$ for $1\leq i \leq k+n_1+\ldots+n_{m-1}.$

 From this multiplication it follows that  the products (\ref{eq3}) are true for $m+1.$

Taking $m=\tau$ in the products (\ref{eq3}) we have the equality (\ref{eq1}) and we complete the proof of the theorem.
\end{proof}

\end{document}